\documentclass[12pt,a4paper]{amsart}
\usepackage{amsmath,amssymb,amsthm} 
\usepackage{enumerate}
\usepackage{fullpage}

\newtheorem{theorem}{Theorem}
\newtheorem{lemma}[theorem] {Lemma}    
\newtheorem{corollary}[theorem] {Corollary}  
\newtheorem{proposition}[theorem] {Proposition}    
\newtheorem{question}[theorem] {Question}

\DeclareMathOperator{\hit}{hit}
\newcommand{\AHM}{\mathcal{A}\mathcal{H}\mathcal{M}}

\begin{document}
\title{Classification of maximum hittings by large families}
\author{Candida Bowtell}
\address{Mathematical Institute, University of Oxford, Radcliffe Observatory Quarter, Woodstock Rd, Oxford OX2 6GG, United Kingdom}
\email{bowtell@maths.ox.ac.uk}
\author{Richard Mycroft}
\address{School of Mathematics, University of Birmingham, Birmingham B15 2TT, United Kingdom}
\email{r.mycroft@bham.ac.uk}
\thanks{This research originated in a summer research placement funded by an award from the LMS Undergraduate Bursary Scheme.}

\date{}

\begin{abstract}
For integers $r$ and $n$, where $n$ is sufficiently large, and for every set $X \subseteq [n]$ we determine the maximal left-compressed intersecting families $\mathcal{A}\subseteq \binom{[n]}{r}$ which achieve maximum hitting with $X$ (i.e. have the most members which intersect $X$). This answers a question of Barber, who extended previous results by Borg to characterise those sets~$X$ for which maximum hitting is achieved by the star.
\end{abstract}

\maketitle

\section {Introduction}

The celebrated Erd\H{o}s-Ko-Rado Theorem~\cite{EKR} states that for all integers $r \leq n/2$ and every family $\mathcal{A} \subseteq \binom{[n]}{r}$, if $\mathcal{A}$ is intersecting (meaning that no two members of $\mathcal{A}$ are disjoint), then $|\mathcal{A}| \leq \binom{n-1}{r-1}$. To see that this bound is tight, fix any $a \in [n]$ and consider the family $\mathcal{S}_a := \{A \in \binom{[n]}{r} : a \in A\}$. We refer to $\mathcal{S}_a$ as the \emph{star at $a$}, and we denote the star at 1 simply by $\mathcal{S}$ (note that $\mathcal{S}_a$ and $\mathcal{S}$ both depend on the values of $n$ and $r$, but this will always be clear from the context). For $r > n/2$ the family $\binom{[n]}{r}$ itself is intersecting, so  the Erd\H{o}s-Ko-Rado Theorem determines the maximum size of an intersecting family on $\binom{[n]}{r}$ for all integers $r$ and $n$.

One natural extension of this result is to find the maximum size of an intersecting family $\mathcal{A} \subseteq \binom{[n]}{r}$ which is \emph{non-trivial}, that is, which is not a subfamily of a star. Hilton and Milner~\cite{HM} demonstrated that in fact such families must be significantly smaller than stars. More precisely, they proved that for all $1<r<n/2$, every non-trivial intersecting family $\mathcal{A} \subseteq \binom{[n]}{r}$ has $|\mathcal{A}| \leq \binom{n-1}{r-1}-\binom{n-r-1}{r-1}+1$. This bound is also tight, as demonstrated by the \emph{Hilton-Milner family} $\mathcal{HM} :=\{A \in \mathcal{S}: A \cap [2, r+1] \neq \emptyset\} \cup [2, r+1]$, and Hilton and Milner additionally proved that (up to isomorphism) $\mathcal{HM}$ is the unique non-trivial intersecting family of this size for $r=2$ and $r \geq 4$, and the families $\mathcal{HM}$ and $\mathcal{A}_{2,3}:=\{A \in \binom{[n]}{r} : |A \cap \{1, 2, 3\}| \geq 2\}$ are the only two non-trivial intersecting families of this size for $r =3$. The logical next step is to ask for the maximum size of an intersecting family $\mathcal{A} \subseteq \binom{[n]}{r}$ which is neither a subfamily of the star nor of the Hilton-Milner family. For $r \geq 4$ this was solved implicitly by Hilton and Milner~\cite{HM}, and very recently Han and Kohayakawa~\cite{HK} gave a simpler proof which also includes the case $r=3$.

The method of \emph{compression} (also known as \emph{shifting}), is a key technique in proving each of the results stated above. Given $i, j \in [n]$ and a family $\mathcal{A} \subseteq \binom{[n]}{r}$, the \emph{$ij$-shift} $S_{ij}(\mathcal{A})$ of $\mathcal{A}$ is the family obtained by the following change: for each set $A \in \mathcal{A}$ for which $i \in A$ and $j \notin A$, replace $A$ by $B := (A \setminus \{i\}) \cup \{j\}$ in $\mathcal{A}$ if $B \notin \mathcal{A}$. We say that a family is \emph{left-compressed} if $S_{ij}(\mathcal{A}) = \mathcal{A}$ for every $i > j$. The following equivalent form of this definition is convenient.
For sets $A, B \in \binom{[n]}{r}$, write $A = \{a_1, \dots, a_r\}$ and $B = \{b_1, \dots, b_r\}$ with $a_1 \leq \dots \leq a_r$ and $b_1 \leq \dots \leq b_r$. We say that $A \leq B$ if $a_i \leq b_i$ for every $i \in [r]$. A family $\mathcal{A} \subseteq \binom{[n]}{r}$ is then left-compressed if for every $A, B \in \binom{[n]}{r}$ with $A \in \mathcal{A}$ and $B \leq A$ we have $B \in \mathcal{A}$. For a wide-ranging overview of compressions of set systems, see the survey by Frankl~\cite{Frankl}.
The relevance to intersecting families arises through the well-known fact that if $\mathcal{A} \subseteq \binom{[n]}{r}$ is intersecting then for every $i, j \in [n]$ the family $S_{ij}(\mathcal{A})$ is also intersecting, so when seeking the maximum size of an intersecting family we can restrict our attention solely to left-compressed families. In particular, it is easily observed that the families $\mathcal{S}$, $\mathcal{HM}$ and $\mathcal{A}_{2, 3}$ are each left-compressed.

Another natural extension of the Erd\H{o}s-Ko-Rado Theorem is to ask for the maximum size of an intersecting family $\mathcal{A} \subseteq \binom{[n]}{r}$ if we only count those sets $A \in \mathcal{A}$ which intersect a fixed subset $X \subseteq[n]$. Without further restriction this problem is a trivial consequence of the Erd\H{o}s-Ko-Rado Theorem (we can fix any $a \in X$ and take $\mathcal{A} = \mathcal{S}_a$), but Borg~\cite{Borg} observed that the `correct' interpretation of the problem is to consider only left-compressed families~$\mathcal{A}$. Using his terminology, we say that a set $A$ \emph{hits} a set $X$ if $A \cap X \neq \emptyset$, and the \emph{hitting} of a family $\mathcal{A}$ with a set $X$ is 
$$\hit_X(\mathcal{A}) := |\{A \in \mathcal{A} : A \cap X \neq \emptyset\}|,$$ that is, the number of members of~$\mathcal{A}$ which hit $X$. So we seek to identify, for each $n, r$ and $X \subseteq [n]$, the left-compressed intersecting families~$\mathcal{A} \in \binom{[n]}{r}$ which maximise $\hit_X(\mathcal{A})$. Clearly we need only consider maximal left-compressed intersecting families (MLCIFs)\footnote{It is important to note that the order of conditions here is irrelevant, in that a maximal left-compressed intersecting family is precisely a maximal intersecting family which is left-compressed. Indeed, if an MLCIF $\mathcal{A}$ is not maximal with respect to the intersecting property, then there is a larger intersecting family $\mathcal{A'}$ with $\mathcal{A} \subseteq \mathcal{A'}$, and by repeated shifts of $\mathcal{A'}$ we obtain a left-compressed intersecting family $\mathcal{A''}$ with $|\mathcal{A''}| = |\mathcal{A'}|$ and $\mathcal{A} \subseteq \mathcal{A''}$, contradicting the maximality of $\mathcal{A}$.}, and we say that an MLCIF is \emph{optimal} for $X$ if it achieves this maximum.

Fix any $1 \leq r \leq n$. If $r > n/2$ then the family $\binom{[n]}{r}$ is the only MLCIF, so vacuously is the unique optimal MLCIF for every $X \subseteq [n]$. We therefore assume henceforth that $n \geq 2r$. Likewise, in the case $r=1$ the family $\{\{1\}\}$ is vacuously the unique optimal MLCIF for every $X \subseteq [n]$. For $r=2$ there exist two MLCIFs, namely $\mathcal{S} = \{\{1,x\} : x \in [2, n]\}$ and $\mathcal{A}_{2, 3} = \{\{1,2\},\{2,3\},\{1,3\}\}$, and a straightforward case analysis shows that $\mathcal{A}_{2, 3}$ is the unique optimal MLCIF for $X \in \{\{2\}$, $\{3\}$, $\{2, 3\}\}$, that both $\mathcal{A}_{2, 3}$ and $\mathcal{S}$ are optimal MLCIFs for $X = \{2, 3, x\}$ with $x \in [4,n]$ and $X = \{y, z\}$ with $y \in \{2, 3\}$ and $z \in [4, n]$, and that for every other non-empty $X$ the family $\mathcal{S}$ is the unique optimal MLCIF. (Actually, if $n=4$ there are a few more cases in which both families are optimal MLCIFs, but we omit the details since our primary interest is the case where $n$ is large.) Unfortunately, for $r \geq 3$ the number of MLCIFs grows rapidly, so case analyses quickly prove intractable. Observe, however, that if $1 \in X$ then the Erd\H{o}s-Ko-Rado Theorem implies that $\mathcal{S}$ is the unique optimal MLCIF for $X$, and if $X$ is empty then trivially every MLCIF is optimal for $X$. We therefore restrict our attention henceforth to non-empty sets $X \subseteq [2, n]$.
Borg considered for which such sets the star is optimal, and gave both general sufficient conditions under which this occurs (Theorem~\ref{borg}), as well as a precise characterisation for the case $|X| = r$ (Theorem~\ref{borg2}).

\begin{theorem}[Borg~\cite{Borg}] \label{borg}
Suppose that $r \geq 2$ and $n \geq 2r$. Then $\mathcal{S}$ is optimal for every $X \subseteq [2,n]$ satisfying at least one of the following:
\begin{enumerate}[(i)]
\item $|X|>r$;
\item $X \geq X'$, where $\mathcal{S}$ is known to be optimal for $X'$;
\item $X=\{2k, 2k+2,\dots, 2r\}$ for any $k \leq r$.
\end{enumerate}
\end{theorem}

\begin{theorem}[Borg~\cite{Borg}] \label{borg2}
Suppose that $r \geq 2$ and $n \geq 2r$, and fix $X \subseteq [2, n]$ with $|X| = r$. 
If $n=2r$, then $\mathcal{S}$ is optimal for $X$ 
if and only if $X \geq \{2,4,\dots,2r\}$, whilst if $n>2r$ then $\mathcal{S}$ is optimal for $X$ 
if and only if one of the following statements holds:
\begin{enumerate}[(i)]
\item $r = 2$ and $X \neq \{2, 3\}$; 
\item $r = 3$ and $|X \cap \{2, 3\}| \leq 1$;
\item $r \geq 4$ and $X \neq [2, r+1]$.
\end{enumerate}
\end{theorem}

More recently Barber~\cite{Barber} generalised these results by precisely characterising the cases for which the star is optimal for sufficiently large $n$. Observe for this that if $X \subseteq [2, r+1]$ is non-empty then $\hit_X(\mathcal{HM}) = \hit_X(\mathcal{S})+1$, so $\mathcal{S}$ is certainly not optimal for such $X$. 
\begin{theorem}[Barber~\cite{Barber}] \label{barber}
Suppose that $r \geq 3$ and that $n$ is sufficiently large, and fix non-empty $X \subseteq [2,n]$. Then $\mathcal{S}$ is optimal for $X$ if and only if $X \not\subseteq [2, r+1]$ and one of the following statements holds:
\begin{enumerate}[(i)]
\item $|X|=1$;
\item $|X|=2$ and $X \cap \{2, 3\} = \emptyset$;
\item $|X|=3$ and $|X \cap \{2, 3\}| \leq 1$;
\item $|X| \geq 4$.
\end{enumerate}
\end{theorem}

Addressing the cases where $\mathcal{S}$ is not optimal for $X$, Barber posed the following question.

\begin{question} \label{barberqu} Is there a short list of families, one of which is optimal for every $X$?
\end{question}

That is, can we find a small collection of MLCIFs $\mathcal{F}$, such that for every $X \subseteq [n]$, there exists $\mathcal{A} \in \mathcal{F}$ such that $\mathcal{A}$ is optimal for $X$? The main result of this paper answers this question in the affirmative for sufficiently large $n$; our list consists of the star $\mathcal{S}$ along with a class of families $\AHM_t$ for $3 \leq t \leq r+1$ which includes the families $\mathcal{A}_{2, 3}$ and $\mathcal{HM}$ introduced previously. 
Specifically, (in the context of fixed integers $n \geq r$,) for each $t \in [n]$ we define $$\AHM_t:=\{A \in \mathcal{S}:A \cap [2,t] \neq \emptyset\} \cup \{A \in \binom{[n]}{r}: [2,t] \subseteq A\}$$
and call $\AHM_t$ the \emph{$t$-adjusted Hilton-Milner} family.
It is easy to check that $\AHM_t$ is a left-compressed intersecting family for every~$t \geq 3$. Furthermore, for $3 \leq t \leq r+1$ the family $\AHM_t$ is in fact an MLCIF  (see Proposition~\ref{AHMtmax}). Observe in particular that $\mathcal{HM}=\AHM_{r+1}$ and that $\mathcal{A}_{2,3}=\AHM_3$. We can now formally state our main result.

\begin{theorem} \label{mainresult}
Suppose that $r \geq 3$ and that $n$ is sufficiently large, and fix a non-empty subset $X \subseteq [2,n]$.  
\begin{enumerate}[(a)]
\item \label{1} If $X=\{2\}$ then $\AHM_3$ is optimal for $X$.
\item \label{2} If $|X|=2$ and $X \cap \{2,3\}\neq \emptyset$ then $\AHM_3$ is optimal for $X$. Furthermore, if also $4 \in X$, then $\AHM_4$ is simultaneously optimal for $X$.
\item \label{3} If $|X|=3$ and $\{2,3\} \subseteq X$ then $\AHM_3$ is optimal for $X$. Furthermore, if also $X=\{2,3,4\}$, then $\AHM_4$ is simultaneously optimal for $X$.
\item \label{4} If $X \subseteq [2,r+1]$ and $X$ is not as in (a)--(c), then $\AHM_m$ is optimal for $X$, where $m := \max X$.
\end{enumerate}
No other MLCIFs are optimal for $X$ as in (a)--(d), and for every other $X \subseteq [2,n]$ the star $\mathcal{S}$ is the unique optimal MLCIF for $X$.
\end{theorem}

In particular, the only non-empty sets $X \subseteq [n]$ for which there is not a unique optimal MLCIF are $\{2, 4\}$, $\{3, 4\}$ and $\{2, 3, 4\}$. Our proof of Theorem~\ref{mainresult} follows the approach of Barber, which in turn developed the work of Ahlswede and Khachatrian~\cite{AK} on generating families. In particular we use Barber's key observation that every MLCIF can be generated by a collection of subsets of $[2r]$ to narrow down the possible MLCIFs for a set $X$ to a collection small enough to compare against each other. We introduce generating families and this key result in the next section, before giving the proof of Theorem~\ref{mainresult} in Section~3 and concluding with some further remarks and questions in Section~4.

\subsection{Notation}

For integers $r \leq n$, we write $[n] := \{1, \dots, n\}$ and $[r, n] := \{r, r+1, \dots, n\}$; for $r > n$ we consider $[r,n]$ to be empty. Given a set $X$ we use $\binom{X}{r}$ to denote the family of all subsets of $X$ of size $r$ and $\mathcal{P}(X)$ to denote the set of all subsets of~$X$.

\section{Generating Families}

Fix integers $1 \leq r \leq n$, and let $\mathcal{G}$ be a collection of subsets of $[n]$. Then the family $\langle \mathcal{G} \rangle _{n,r}$ {\it generated by $\mathcal{G}$ with respect to $n$ and $r$} is given by $\langle \mathcal{G} \rangle_{n,r}:= \{ A \in \binom{[n]}{r} : A\supseteq G \mbox{~for some~} G \in \mathcal{G}\}$; we omit the subscripts and write simply $\langle \mathcal{G} \rangle$ when $n$ and $r$ are clear from the context. We call $\mathcal{G}$ a {\it generating family} of $\langle \mathcal{G} \rangle$. Observe that members of $\mathcal{G}$ of size greater than $r$ do not contribute to $\langle \mathcal{G} \rangle$. Every family $\mathcal{A} \in \binom{[n]}{r}$ is a generating family of itself, but many families $\mathcal{A}$ admit more concise generating families. For example, we have $\mathcal{S}=\langle \{\{1\}\} \rangle$, $\mathcal{A}_{2,3}=\langle \{\{1,2\},\{1,3\},\{2,3\}\} \rangle$, $\mathcal{HM}=\langle \{\{1,i\}: 2 \leq i \leq r+1\} \cup \{[2,r+1]\} \rangle$ and $\AHM_t=\langle\{\{1,i\}:2 \leq i \leq t\} \cup \{[2,t]\}\rangle$. 

The following key observation of Ahlswede and Khachatrian motivates this definition for working with intersecting families.

\begin{theorem}[Ahlswede-Khachatrian \cite{AK}] \label{intersecting} Suppose that $n \geq 2r$ and that $\mathcal{G} \subseteq \mathcal{P}([n])$ has $|G| \leq r$ for every $G \in \mathcal{G}$. Then $\mathcal{G}$ is intersecting if and only if $\langle \mathcal{G} \rangle$ is intersecting.
\end{theorem}

Since there may be many different generating families for an MLCIF on $\binom{[n]}{r}$, it is helpful to define a single canonical generating family of each such family. For an MLCIF $\mathcal{A} \subseteq \binom{[n]}{r}$ we do this as follows. First, we say that a set $G \subseteq [n]$ is a \emph{potential generator} of $\mathcal{A}$ if for every $A \in \binom{[n]}{r}$ with $G \subseteq A$ we have $A \in \mathcal{A}$. We then define the \emph{canonical generating family} $\mathcal{G}$ of $\mathcal{A}$ to be the set of all minimal potential generators of $\mathcal{A}$ (where minimality is with respect to inclusion), and we call the elements of $\mathcal{G}$ the \emph{generators} of~$\mathcal{A}$. Observe that since every element of $\mathcal{A}$ is a potential generator of $\mathcal{A}$, the canonical generating family $\mathcal{G}$ of $\mathcal{A}$ is indeed a generating family of $\mathcal{A}$. Also note that by definition $\mathcal{G}$ must be an \emph{antichain}, meaning that no element of $\mathcal{G}$ is a proper subset of another element of $\mathcal{G}$. Our next proposition establishes the key property that $\mathcal{G}$ is supported on the first $2r$ elements of~$[n]$, and is in fact essentially unique in having this property (the existence of a generating family with this property can also be obtained from results of Barber~\cite{Barber}; see Lemma 8 and the discussion preceding it).

\begin{lemma} \label{canon}
Fix integers $n \geq r$, let $\mathcal{A}$ be an MLCIF on $\binom{[n]}{r}$, and let $\mathcal{G}$ be the canonical generating family of $\mathcal{A}$. Then $G \subseteq [2r]$ for every $G \in \mathcal{G}$. Furthermore, if $n \geq 3r$ then~$\mathcal{G}$ is the only generating family of $\mathcal{A}$ which is an antichain each of whose members is a subset of $[2r]$.
\end{lemma}

\begin{proof}
To prove the first statement, suppose for a contradiction that there exists $G \in \mathcal{G}$ with $G \not\subseteq [2r]$. Then the set $X := G \cap [2r]$ is a proper subset of $G$. Let $A$ be the set consisting of the elements of $G$ and the $r-|G|$ largest elements of $[n] \setminus G$, so $|A| = r$ and we have $A \in \mathcal{A}$ since $G \subseteq A$ and $G$ is a generator of $\mathcal{A}$. Furthermore, since $G$ is a minimal potential generator of $\mathcal{A}$, the set $X$ is not a potential generator of $\mathcal{A}$, meaning that there exists a set $B \in \binom{[n]}{r} \setminus \mathcal{A}$ with $X \subseteq B$. Now, since $\mathcal{A}$ is a maximal intersecting family, there must exist a set $C \in \mathcal{A}$ with $C \cap B = \emptyset$ (as otherwise we could add $B$ to~$\mathcal{A}$). It follows that $C \cap X = \emptyset$. Choose any set $Z \subseteq [2r] \setminus (X \cup C)$ with $|Z| = r-|X|$ (this is possible since $|C| = r$ so $[2r] \setminus (X \cup C)$ has size at least $r-|X|$). Then $D := X \cup Z$ is a set of size $r$. Moreover our choices of $A$ and $Z$ ensure that $D \leq A$, so the fact that $A \in \mathcal{A}$ and $\mathcal{A}$ is left-compressed implies $D \in \mathcal{A}$. However, $D \cap C = \emptyset$, contradicting the fact that $\mathcal{A}$ is intersecting. 

For the second statement, since $\mathcal{G}$ is an antichain, it suffices to prove that for $n \geq 3r$ there do not exist two distinct generating families $\mathcal{G}_1$ and $\mathcal{G}_2$ of $\mathcal{A}$ which are both antichains such that every $G \in \mathcal{G}_1 \cup \mathcal{G}_2$ has $G \subseteq [2r]$. Suppose for a contradiction that such families exist, and let $i$ be minimal such that $\mathcal{G}_1 \cap \binom{[2r]}{i} \neq \mathcal{G}_2 \cap \binom{[2r]}{i}$. Assume without loss of generality that there exists $A \in \mathcal{G}_1 \cap \binom{[2r]}{i}$ with $A \notin \mathcal{G}_2$. 
Since $A \in \mathcal{G}_1$ we have $T := A \cup \{n-r+i+1, \dots, n\} \in \mathcal{A}$, so there must exist $B \in \mathcal{G}_2$ with $B \subseteq T$. However, since $A \notin \mathcal{G}_2$ we have $B \neq A$, whilst by minimality of $i$ and the fact that $\mathcal{G}_1$ is an antichain we cannot have $B \subsetneq A$. It follows that $B \not\subseteq A$, that is, $B \cap \{n-r+i+1, \dots, n\} \neq \emptyset$. However, for $n \geq 3r$ we then have $B \not\subseteq [2r]$, contradicting our assumption on $\mathcal{G}_2$.
\end{proof}

We define the \emph{rank} of an MLCIF $\mathcal{A} \subseteq \binom{[n]}{r}$ to be the smallest size of a generator of~$\mathcal{A}$ (this is well-defined since the generators are the members of the canonical generating family). Clearly $\mathcal{S}$ is the unique MLCIF of rank one. The following proposition plays a key role in the proof of our main theorem by showing that when identifying optimal MLCIFs for a non-empty set $X$ we need only consider MLCIFs of rank one or two; MLCIFs of larger rank simply cannot generate enough sets to be optimal.

\begin{proposition} \label{nocontenders}
Fix $r$ and let $n$ be sufficiently large. For every non-empty $X \subseteq [2,n]$, every MLCIF which is optimal for $X$ has rank one or two. 
\end{proposition}

\begin{proof}
Fix a non-empty set $X \subseteq [2, n]$ and let $\mathcal{A}$ be an MLCIF of rank at least three.
Then by Lemma~\ref{canon} there are at most $2^{2r}$ generators in $\mathcal{G}$, each of which generates at most $\binom{n}{r-3}$ elements of $\mathcal{A}$, so $\hit_X(\mathcal{A}) \leq 2^{2r} \binom{n}{r-3}$. On the other hand, the star $\mathcal{S}$ is an MLCIF with $\hit_X(\mathcal{S}) \geq \binom{n-2}{r-2}$, so for $n$ sufficiently large $\mathcal{A}$ is not optimal for~$X$. 
\end{proof}

Our next lemma shows that the canonical generating family of an MLCIF $\mathcal{A}$ partially inherits the property of being left-compressed, in the sense that the family of generators of smallest size must be left-compressed. Combined with Theorem~\ref{intersecting} this shows that in fact these generators form a left-compressed intersecting family, though not necessarily an MLCIF, as shown e.g.~by $\AHM_4$.

\begin{lemma} \label{compressed}
Fix $n \geq 2r$, let $\mathcal{A}$ be an MLCIF on $\binom{[n]}{r}$, let $\mathcal{G}$ be the canonical generating family of $\mathcal{A}$, and let $k$ be the rank of $\mathcal{A}$. Then the subfamily $\mathcal{G} \cap \binom{[n]}{k}$ is left-compressed.
\end{lemma}

\begin{proof}
Suppose for a contradiction that there exist $A \in \mathcal{G} \cap \binom{[n]}{k}$ and $B \in \binom{[n]}{k} \setminus \mathcal{G}$ with $B \leq A$. Let $C$ be the set of the $r-k$ largest elements of $[n] \setminus A$, and let $D$ be the set of the $r - k$ largest elements of $[n] \setminus B$. Then $S := A \cup C$ and $T := B \cup D$ are both elements of $\binom{[n]}{r}$, and the fact that $B \leq A$ implies that $T \leq S$. Since $A \subseteq S$ and $A \in \mathcal{G}$ we have $S \in \mathcal{A}$, and since $\mathcal{A}$ is left-compressed it follows that $T \in \mathcal{A}$. However, the fact that $\mathcal{A}$ is left-compressed then implies that $B \cup F \in \mathcal{A}$ for every set $F \in \binom{[n] \setminus B}{r-k}$, and so $B$ is a potential generator of $\mathcal{G}$. This gives a contradiction, since $B \notin \mathcal{G}$ and $\mathcal{A}$ has no generators of size less than $k=|B|$. 
\end{proof}

We now justify our assertion made in the introduction that the family $\AHM_t$ is indeed an MLCIF.

\begin{proposition} \label{AHMtmax}
For $n \geq 2r$ and $3 \leq t \leq r+1$, the family $\AHM_t$ is an MLCIF.
\end{proposition}

\begin{proof}
For $t \geq 3$, the fact that $\AHM_t$ is left-compressed follows immediately from the definition, and Theorem~\ref{intersecting} implies that $\AHM_t$ is intersecting. 
It remains to show that $\AHM_t$ is maximal with these properties. For $t=r+1$ this follows from the Hilton-Milner Theorem~\cite{HM} which states that $\mathcal{HM} = \AHM_{r+1}$ is the largest intersecting family which is not a subfamily of a star. So suppose for a contradiction that $t \leq r$ and $\AHM_t$ is not an MLCIF. Then $\AHM_t$ is a proper subset of an MLCIF $\AHM_t^{*}$, so we may choose a set $A \in \AHM_t^{*} \setminus \AHM_t$. It follows from the definition of $\AHM_t$ that if $1 \in A$ then $\{1,t+1, t+2,\dots,t+r-1\} \leq A$, and if $1 \notin A$ then $\{2,3,\dots,t-1,t+1,t+2,\dots,r+2\} \leq A$.
Since $\AHM_t^{*}$ is left-compressed this implies that either $\{1,t+1, t+2,\dots,t+r-1\} \in \AHM_t^{*}$ or $\{2,3,\dots,t-1,t+1,t+2,\dots,r+2\} \in \AHM_t^{*}$.
Observe that $\{2,3,\dots,t,t+r,t+r+1,\dots,2r\} \in \AHM_t$ and $\{1,t,r+3,r+4,\dots,2r\} \in \AHM_t$ but $\{1,t+1, t+2,\dots,t+r-1\} \cap \{2,3,\dots,t,t+r,t+r+1,\dots,2r\} = \emptyset$, and $\{2,3,\dots,t-1,t+1,t+2,\dots,r+2\} \cap \{1,t,r+3,r+4,\dots,2r\} = \emptyset$. So in either case the family $\AHM_t^{*}$ is not intersecting, a contradiction.
\end{proof}

Observe that if the set $\{2, 3\}$ is a generator of an MLCIF $\mathcal{A}$, then $\AHM_3 \subseteq \mathcal{A}$, and it then follows from the maximality of $\AHM_3$ that $\mathcal{A} = \AHM_3$. This establishes the following corollary.

\begin{corollary} \label{23max}
For $n \geq 2r$, if $\{2, 3\}$ is a generator of an MLCIF $\mathcal{A}$, then $\mathcal{A}=\AHM_3$.
\end{corollary}

Using this, we can establish a more detailed understanding of MLCIFs of rank 2. For this we define $\mathcal{I}_j$ for $j \geq 2$ to be the set of all MLCIFs $\mathcal{A} \subseteq \binom{[n]}{r}$ of rank two whose generators of size two are precisely the sets $\{1,2\} ,\dots, \{1,j\}$. Observe that $\AHM_m \in \mathcal{I}_m$ for every $m \geq 4$, but that $\AHM_3 \notin \mathcal{I}_3$. 

\begin{proposition} \label{size2gen}
Let $n \geq 2r$ and suppose that $\mathcal{A} \subseteq\binom{[n]}{r}$ is an MLCIF of rank $2$. Then either $\mathcal{A}=\AHM_3$ or $\mathcal{A} \in \bigcup_{j=2}^{r+1} \mathcal{I}_j$.
\end{proposition}

\begin{proof} 
Let $\mathcal{F}$ be the set of all generators of $\mathcal{A}$ of size two. If $\{2, 3\} \in \mathcal{F}$ then $\mathcal{A} = \AHM_3$ by Corollary~\ref{23max}, so we may assume $\{2, 3\} \notin \mathcal{F}$. Since $\mathcal{F}$ is left-compressed by Lemma~\ref{compressed} it follows that $\mathcal{F} = \{\{1, i\}: 2 \leq i \leq j\}$ for some integer $j \geq 2$, that is, that $\mathcal{A} \in \mathcal{I}_j$. If $j > r+1$ then the fact that $\mathcal{A}$ is intersecting implies that $1 \in A$ for every $A \in \mathcal{A}$, so $\mathcal{A}$ is a subfamily of the star~$\mathcal{S}$, contradicting the fact that $\mathcal{A}$ is an MLCIF of rank $2$. So we must have $j \leq r+1$ as required. 
\end{proof}

\section{Proof of Theorem~\ref{mainresult}}

Proposition~\ref{nocontenders} tells us that every MLCIF which is 
optimal for any non-empty $X \subseteq [2,n]$ must have rank one or two. 
Before proceeding to the proof of Theorem~\ref{mainresult} we now 
further narrow down these possibilities to just two MLCIFs for each such 
set $X \neq \{2\}$. These possibilities are given in 
Corollary~\ref{ahmopt}, which follows directly from our next proposition 
stating that almost all members of $\bigcup_{j = 2}^{r+1} \mathcal{I}_j$ 
cannot be optimal. Similar statements can be made for the case $X = 
\{2\}$, but due to the fact that $\AHM_2$ is not an MLCIF it is 
convenient instead to defer this case to the proof of 
Theorem~\ref{mainresult}.

\begin{proposition} \label{ahmbest}
Fix $r \geq 3$, let $n$ be sufficiently large, let $X \subseteq [2, n]$ be non-empty and write $m := \max X$. 
\begin{enumerate}[(i)]
\item If $X \not\subseteq [2, r+1]$ then $\hit_X(\mathcal{S}) > \hit_X(\mathcal{A})$ for every $\mathcal{A} \in \bigcup_{j=2}^{r+1} \mathcal{I}_j$. That is, $\mathcal{S}$ hits $X$ more than any family in $\bigcup_{j=2}^{r+1} \mathcal{I}_j$.
\item If $X \subseteq [2, r+1]$ and $X \neq \{2\}$, then $\hit_X(\AHM_m) > \hit_X(\mathcal{A})$ for every $\mathcal{A} \in \bigcup_{j=2}^{r+1} \mathcal{I}_j \setminus \{\AHM_m\}$. That is, $\AHM_m$ hits $X$ more than any other family in $\bigcup_{j=2}^{r+1} \mathcal{I}_j$.
\end{enumerate}
\end{proposition}

\begin{proof}
For (i), fix an MLCIF $\mathcal{A} \in \bigcup_{j=2}^{r+1} \mathcal{I}_j$ and let $\mathcal{G}$ be the canonical generating family of $\mathcal{A}$. By Lemma~\ref{canon} 
we have~$|\mathcal{G}| \leq 2^{2r}$. Define
$$\mathcal{S}':=\{S \in \mathcal{S}:\{1,m\} \subseteq S \mbox{~and~} [2,r+1] \cap S =\emptyset\},$$ 
and
$$\mathcal{A}^{\star}:=\left\{A \in \mathcal{A}:\{1,j\}\subseteq A \mbox{ for some $j \in [2, r+1]$}\right\}.$$
Then (for sufficiently large $n$) we have $|\mathcal{S}'| = \binom{n-r-2}{r-2} > 2^{2r}\binom{n}{r-3}  \geq |\mathcal{A} \setminus \mathcal{A}^{\star}|$; the final inequality holds since every set in $\mathcal{A} \setminus \mathcal{A}^{\star}$ is generated by one of the at most $2^{2r}$ generators of size at least 3, each of which generates at most $\binom{n}{r-3}$ sets. Observe also that $\mathcal{A}^{\star} \subseteq \mathcal{S}$, that $\mathcal{S}' \cap \mathcal{A}^{\star} = \emptyset$, and that every element of $\mathcal{S}'$ is an element of $\mathcal{S}$ which hits $X$. It follows that
\begin{align*}
\hit_X(\mathcal{S}) &\geq \hit_X(\mathcal{A}^{\star})+ |\mathcal{S}'| 
> \hit_X(\mathcal{A}^{\star})+ |\mathcal{A} \setminus \mathcal{A}^{\star}| \geq \hit_X(\mathcal{A}),
\end{align*}
as required.

For (ii) we introduce the following notation: for any MLCIF $\mathcal{A}$, write 
$$\mathcal{A}^\circ := \left\{A \in \mathcal{A} : \{1, j\} \subseteq A \mbox{ for some $j \in [2, m]$}\right\} \mbox{ and } \mathcal{A}^+ = \mathcal{A} \setminus \mathcal{A}^\circ.$$
Assume $X \subseteq [2, r+1]$, and observe that since $m = \max X$, for each $x \in X$ the set $\{1, x\}$ is a generator of $\AHM_m$. It follows that $\hit_X(\AHM_m) \geq |X|\binom{n-r-1}{r-2}$. Consider any MLCIF $\mathcal{A} \in \bigcup_{j=2}^{r+1} \mathcal{I}_j$ with $\mathcal{A} \neq \AHM_m$, and let $\mathcal{G}$ be the canonical generating family of~$\mathcal{A}$, so $|\mathcal{G}| \leq 2^{2r}$ by Lemma~\ref{canon}. Suppose first that $\mathcal{A} \in \bigcup_{j=2}^{m-1} \mathcal{I}_j$. Then $\mathcal{G}$ contains at most $|X|-1$ generators of the form $\{1, x\}$ with $x \in X$, each of which generates at most $\binom{n}{r-2}$ members of $\mathcal{A}$, whilst each of the at most $2^{2r}$ remaining generators $G \in \mathcal{G}$ generates at most $r\binom{n}{r-3}$ members of $\mathcal{A}$ which hit $X$. (In fact, generators $G$ satisfying $|G|=2$ but not hitting $X$ generate at most $r\binom{n}{r-3}$ members of $\mathcal{A}$ which hit $X$, namely those sets which contain both $G$ and one of the at most $r$ elements of $X$.
All other generators have size at least $3$, and thus generate at most $\binom{n}{r-3}$ members of~$\mathcal{A}$.) 
So we have $\hit_X(\mathcal{A}) \leq (|X|-1)\binom{n}{r-2} + 2^{2r}r\binom{n}{r-3}$, and so (for $n$ sufficiently large) $\hit_X(\AHM_m) > \hit_X(\mathcal{A})$, as required. 

We may therefore assume that $\mathcal{A} \in \bigcup_{j=m}^{r+1} \mathcal{I}_j$, and in particular that $\{1, j\}$ is a generator of $\mathcal{A}$ for every $j \in [2, m]$. Observe that we then have $\mathcal{A}^\circ = \AHM_m^\circ$, so $\hit_X(\mathcal{A}^\circ) = \hit_X(\AHM_m^\circ)$. We now compare $\hit_X(\mathcal{A}^+)$ and $\hit_X(\AHM_m^+)$; observe for this that $\AHM_m^+$ contains precisely those sets $S \in \binom{[n]}{r}$ with $1 \notin S$ and $[2, m] \subseteq S$, so we have $\hit_X(\AHM_m^+) = \binom{n-m}{r-m+1}$. On the other hand, since $\mathcal{A}$ is intersecting, Theorem~\ref{intersecting} tells us that $\mathcal{G}$ is intersecting also; since $\mathcal{G}$ includes $\{1, j\}$ for every $j \in [2, m]$, it follows that every generator $G\in \mathcal{G}$ must satisfy either $1 \in G$ or $[2, m] \subseteq G$. 
However, every set $A \in \mathcal{A}$ with $1 \in A$ which hits $X$ is an element of $\mathcal{A}^\circ$, so the sets generated by generators $G$ with $1 \in G$ do not contribute to $\hit_X(\mathcal{A}^+)$. Also, since $\AHM_m$ is an MLCIF whose generators are $[2, m]$ and $\{1, j\}$ for $j \in [2, m]$, and $\mathcal{A} \neq \AHM_m$, we must have $[2, m] \notin \mathcal{G}$. 
So every generator $G \in \mathcal{G}$ with $[2, m] \subseteq G$ has size at least $m$, and so the number of sets generated by generators of this form is at most $2^{2r}\binom{n}{r-m}$. We conclude that (for sufficiently large~$n$) 
we have $\hit_X(\mathcal{A}^+) \leq 2^{2r}\binom{n}{r-m} < \hit_X(\AHM_m^+)$, and so
$$\hit_X(\mathcal{A}) = \hit_X(\mathcal{A}^\circ) + \hit_X(\mathcal{A}^+) < \hit_X(\AHM_m^\circ) + \hit_X(\AHM_m^+) = \hit_X(\AHM_m),$$
as required.
\end{proof}

Recall that if $X \subseteq [2, r+1]$ is non-empty then the star $\mathcal{S}$ is not optimal for $X$ since $\hit_{X}(\AHM_{r+1}) = \hit_X(\mathcal{S}) +1$. This fact, together with Propositions~\ref{nocontenders},~\ref{size2gen} and~\ref{ahmbest}, immediately implies the following important corollary, narrowing down the list of potential optimal families for a set $X$ to just two possibilities.

\begin{corollary} \label{ahmopt}
For every $r \geq 3$ the following statements hold for sufficiently large~$n$.
\begin{enumerate}
\item For every non-empty $X \subseteq [2,r+1]$ with $X \neq \{2\}$, if $\mathcal{A}$ is an MLCIF which is optimal for $X$ then $\mathcal{A} \in \{\AHM_3, \AHM_m\}$, where $m = \max X$.
\item For every $X \not\subseteq [2, r+1]$, if $\mathcal{A}$ is an MLCIF which is optimal for $X$ then $\mathcal{A} \in \{\mathcal{S}, \AHM_3\}$. 
\end{enumerate}
\end{corollary} 

Finally, to prove Theorem~\ref{mainresult} we simply need to compare, for each set $X$, the hitting of these two potential optimal families with $X$. We do this on a case-by-case basis.

\begin{proof}[Proof of Theorem~\ref{mainresult}]
Throughout this proof we will use our assumption that $n$ is sufficiently large relative to $r$ without further comment. We begin with case~\eqref{1}, where $X = \{2\}$. By Proposition~\ref{nocontenders} and Proposition~\ref{size2gen} the only possible optimal MLCIFs for $X$ are $\mathcal{S}$, $\mathcal{\AHM}_3$ and the members of $\mathcal{I}_j$ for $2 \leq j \leq r+1$. Observe that $\hit_X(\AHM_3) = \binom{n-2}{r-2} + \binom{n-3}{r-2}$, whilst $\hit_X(\mathcal{S}) = \binom{n-2}{r-2}$. Furthermore, for each $\mathcal{A} \in \bigcup_{j=2}^{r+1} \mathcal{I}_j$ we have $\hit_X(\mathcal{A}) \leq \binom{n-2}{r-2} + 2^{2r}\binom{n}{r-3}$, since $\mathcal{A}$ has at most $2^{2r}$ generators by Lemma~\ref{canon}. It follows that $\AHM_3$ is the unique optimal MLCIF for $X$. 

We next turn to case~\eqref{2}, where $|X| = 2$ and $X \cap \{2, 3\} \neq \emptyset$. If $X=\{2,3\}$ then Corollary~\ref{ahmopt} implies that $\AHM_3$ is the unique optimal MLCIF for $X$. Assume therefore that either $X = \{2,m\}$ or $X = \{3,m\}$ for some $m \geq 4$. In each case we have
\begin{align*}
\hit_X(\mathcal{S}) &= \binom{n-2}{r-2}+\binom{n-3}{r-2},\\
\hit_X(\AHM_3) &= \binom{n-2}{r-2}+\binom{n-3}{r-2}+\binom{n-4}{r-3}, \mbox{ and, if $m \in [4,r+1]$, then}\\
\hit_X(\AHM_m) &= \binom{n-2}{r-2}+\binom{n-3}{r-2}+\binom{n-m}{r-m+1}.
\end{align*}
Note that $\binom{n-4}{r-3} \geq \binom{n-m}{r-m+1}$ for all $m \geq 4$ with equality if and only if $m=4$. By Corollary~\ref{ahmopt} it follows that $\AHM_3$ is the unique optimal MLCIF for $X$ for $m > 4$, whilst $\AHM_3$ and $\AHM_4$ are the only two optimal MLCIFs for $X$ if $m = 4$.

Now we consider case~\eqref{3}, where $X=\{2,3,m\}$ for some $m \geq 4$. We then have 
\begin{align*}
\hit_X(\mathcal{S}) &= \binom{n-2}{r-2} + \binom{n-3}{r-2} + \binom{n-4}{r-2},\\
\hit_X(\AHM_3) &=\binom{n-2}{r-2}+\binom{n-3}{r-2}+\binom{n-3}{r-2}, \mbox{ and, if $m \in [4,r+1]$, then}\\
\hit_X(\AHM_m) &= \binom{n-2}{r-2}+\binom{n-3}{r-2}+\binom{n-4}{r-2}+\binom{n-m}{r-m+1}.
\end{align*}
Observe that $\binom{n-3}{r-2} \geq \binom{n-4}{r-2}+\binom{n-m}{r-m+1}$ for all $m \geq 4$ with equality holding if and only if $m=4$. By Corollary~\ref{ahmopt} it follows that $\AHM_3$ is the unique optimal MLCIF for $X$ for $m > 4$ whilst $\AHM_3$ and $\AHM_4$ are the only two optimal MLCIFs for $X$ if $m = 4$.

Lastly, in case~\eqref{4} we have that $X \subseteq [2, r+1]$ and that $X$ does not meet the conditions of cases~\eqref{1}--\eqref{3}. Define $m := \max X$. Then by Corollary~\ref{ahmopt} the only two possibilities for optimal MLCIFs for $X$ are $\AHM_3$ and $\AHM_m$. Observe that $\AHM_m$ has $|X|$ generators of size $2$ which intersect $X$, so $\hit_X(\AHM_m) \geq |X| \binom{n-r-1}{r-2}$. 
If $\{2,3\} \subseteq X$, then $|X| \geq 4$  (otherwise we have case~\eqref{2} or~\eqref{3}), so we have
\[\hit_X(\AHM_3) \leq 3\binom{n-2}{r-2} < |X| \binom{n-r-1}{r-2} \leq \hit_X(\AHM_m).\] 
Similarly, if $|\{2,3\} \cap X|=1$ then $X = \{3\}$ or $|X| \geq 3$ (otherwise we have case~\eqref{1} or~\eqref{2}). If $X = \{3\}$ then $\AHM_3 = \AHM_m$, whilst if $|X| \geq 3$ then 
we have
\[\hit_X(\AHM_3) \leq 2\binom{n-2}{r-2} + |X| \binom{n}{r-3} < |X| \binom{n-r-1}{r-2} \leq \hit_X(\AHM_m).\] 
When $X \cap \{2,3\} = \emptyset$, the family $\AHM_3$ has no rank $2$ generators hitting $X$, whilst $\AHM_m$ has $|X|$ such generators, so certainly $\hit_X(\AHM_3) < \hit_X(\AHM_m)$. In each case it follows that $\AHM_m$ is the unique optimal MLCIF for $X$.
 
Finally, it remains to prove that the star is the unique optimal MLCIF for every set $X \subseteq [2, n]$ which is not covered by cases~\eqref{1}--\eqref{4}. Any such $X$ has $X \not \subseteq [2, r+1]$, so by Corollary~\ref{ahmopt} it suffices for this to show that $\hit_X(\AHM_3) < \hit_X(\mathcal{S})$. Moreover, any such $X$ satisfies either
\begin{enumerate}[(i)]
\item $|X|=1$,
\item $|X|=2$ and $X \cap \{2, 3\} = \emptyset$,
\item $|X|=3$ and $|X \cap \{2, 3\}| \leq 1$, or
\item $|X| \geq 4$.
\end{enumerate}
Observe that in cases~(i), (ii) and~(iii) we have $\hit_X(\mathcal{S}) \geq |X| \binom{n-4}{r-2}$. However, in cases~(i) and~(ii) we also have $\hit_X(\AHM_3) \leq 6\binom{n}{r-3}$, and in case~(iii) we have $\hit_X(\AHM_3) \leq 2\binom{n}{r-2} + 2\binom{n}{r-3}$. Similarly in case (iv) we have $\hit_X(\mathcal{S}) \geq 4 \binom{n-5}{r-2}$ and $\hit_X(\AHM_3) \leq 3\binom{n}{r-2}$. So in all cases we have $\hit_X(\AHM_3) < \hit_X(\mathcal{S})$, as required.
\end{proof}

We finish this section by returning to the question of which left-compressed intersecting families (LCIFs) have maximum hitting with a fixed non-empty set $X \subseteq [n]$. For this we extend the definition of optimality to LCIFs in the natural way, saying that an LCIF $\mathcal{A} \subseteq \binom{[n]}{r}$ is optimal for $X$ if $\hit_X(\mathcal{A}) \geq \hit_X(\mathcal{F})$ for every LCIF $\mathcal{F} \subseteq \binom{[n]}{r}$. As for MLCIFs, if $1 \in X$ then $\mathcal{S}$ is the unique optimal LCIF, so again we consider only $X \subseteq [2, n]$. Since every LCIF is a subfamily of an MLCIF, the optimal LCIFs for $X$ are precisely the left-compressed subfamilies of optimal MLCIFs which can be formed by removing sets which do not hit~$X$. From this observation we obtain the following corollary (which should be read in conjunction with Theorem~\ref{mainresult}). 

\begin{corollary} \label{best}
Let $r \geq 3$ and $n$ be sufficiently large. Suppose that $X \subseteq [2,n]$ is non-empty and let $m:= \max X$. 
\begin{enumerate}[(i)]
\item If $\mathcal{S}$ is not an optimal MLCIF for $X$ then the optimal LCIFs for $X$ are precisely the optimal MLCIFs for $X$.
\item If $\mathcal{S}$ is an optimal MLCIF for $X$ then the optimal LCIFs for $X$ are precisely the LCIFs $\mathcal{A}$ with $\AHM_m \subseteq \mathcal{A} \subseteq \mathcal{S}$. 
\end{enumerate}
\end{corollary}

\begin{proof}
Suppose first that $\mathcal{S}$ is not an optimal MLCIF for $X$. Then by Theorem~\ref{mainresult} every optimal MLCIF for $X$ has the form $\AHM_t$ for some $t \in [3, r+1]$, and moreover we have $t \in X$ in every case except when $X = \{2\}$ or $X = \{2,x\}$ with $x \in [4, n]$, in which case $t=3$. (When $X = \{2,4\}$ both $\AHM_3$ and $\AHM_4$ are optimal for $X$; for the former we have $t=3$ and for the latter we have $t \in X$.) Observe that every set $A \in \AHM_t$ has either $A \leq B := \{1, t, n-r+3, \dots, n\}$ or $A \leq C := \{2, 3, \dots, t, n-r+t, \dots, n\}$, and furthermore that for $t=3$ every set $A \in \AHM_3$ has $A \leq C$. Since $C$ hits $X$ in all cases, and $B$ hits $X$ if $t \in X$, it follows that every LCIF $\mathcal{A}$ which is a proper subfamily of $\AHM_t$ has $\hit_X(\mathcal{A}) < \hit_X(\mathcal{\AHM}_t)$, and so is not optimal, proving~(i).

Now suppose that $\mathcal{S}$ is an optimal MLCIF for $X$. Then $\mathcal{S}$ is the unique optimal MLCIF for $X$ by Theorem~\ref{mainresult}, so every optimal LCIF $\mathcal{A}$ for $X$ has $\mathcal{A} \subseteq \mathcal{S}$. Furthermore we have $m > r+1$, so $\AHM_m$ consists precisely of those sets $A \in \binom{[n]}{r}$ with $A \leq D$, where~$D$ is the set formed by adding the $r-2$ largest elements of $[n] \setminus \{m\}$ to $\{1, m\}$. Since $D$ hits $X$ it follows that every optimal LCIF $\mathcal{A}$ has $\AHM_m \subseteq \mathcal{A}$, and~(ii) follows since $\hit_X(\mathcal{\AHM}_m) = \hit_X(\mathcal{S})$.\footnote{Since here we have $m > r+1$, the family $\AHM_m$ is not an MLCIF, so this conclusion does not contradict our assertion that $\mathcal{S}$ is the unique optimal MLCIF for $X$.}
\end{proof}

\section{Further Directions}

It would be interesting to know how large $n$ must be to satisfy Theorem~\ref{mainresult} (Barber previously asked the analogous question following his proof of Theorem~\ref{barber}). Following our proofs directly gives a bound on $n$ which is exponential in $r$, but we suspect that more careful arguments would yield a polynomial bound.

Recall that, for sufficiently large $n$, Theorem~\ref{barber} identified all $X \subseteq [n]$ for which an MLCIF of rank $1$ (that is, $\mathcal{S}$) is optimal, and Theorem~\ref{mainresult} shows that in all other cases every optimal MLCIF for $X$ has rank $2$. In the spirit of the Hilton-Milner theorem, it would also be interesting to consider the optimal MLCIF among all families other than the star $\mathcal{S}$, giving the following question.

\begin{question} \label{notstar}
For each $n \geq 2r$ and $X \subseteq [n]$, which MLCIFs $\mathcal{T} \neq \mathcal{S}$ satisfy $\hit_X(\mathcal{T}) \geq \hit_X(\mathcal{A})$ for every MLCIF $\mathcal{A} \neq \mathcal{S}$?
\end{question}

To answer Question~\ref{notstar} we must certainly consider MLCIFs of rank greater than $2$. Indeed, by Proposition~\ref{size2gen} every MLCIF of rank $2$ has no size $2$ generators hitting any element $x \in X$ such that $x > r+1$. So, for example, when $r=3$ and $X=\{5\}$, no generator of size $2$ in a canonical generating family can hit $X$. Observe that the family $\mathcal{A}_{3,4,5}:= \langle\{\{a,b,c\}: 1 \leq a < b  < c \leq 5\}\rangle$ has $6$ generators of size $3$ hitting $X$. Every other MLCIF (excluding the star) has at most $5$ generators of size $3$ hitting $X$, and thus for sufficiently large $n$ the family $\mathcal{A}_{3,4,5}$ is unique in achieving maximum hitting with $X$ among all MLCIFs excluding the star. The problem appears to become significantly harder for larger values of $r$, for which it seems difficult just to enumerate all the MLCIFs which exist. In fact it seems to be non-trivial to resolve even the apparently-simpler question of identifying, for every $X \subseteq [n]$, the MLCIFs $\mathcal{A}$ which maximise $\hit_X(\mathcal{A})$ among all MLCIFs of rank two.

\section*{Acknowledgements}

We thank Ben Barber for helpful discussions, and two anonymous reviewers for their helpful suggestions for improving the presentation of this manuscript.

\end{document}